\documentclass[reqno]{amsart}
\usepackage{amsfonts, amsmath, amsthm, amssymb, latexsym, graphicx}

\usepackage[applemac]{inputenc}%otherwise: latin1
\usepackage{amsthm}
\usepackage{amssymb}
\usepackage{amsmath}
\usepackage{graphicx}
\usepackage{eucal}  % Kalligraphie
\usepackage{mathrsfs}
\usepackage{mathtools}
\usepackage{nicefrac}
\usepackage{stackrel}

\usepackage{ifthen}
\newboolean{skript}
\newcommand{\beweis}{\ifthenelse{\boolean{skript}} {
\begin{proof}[Beweis]
\mbox{}\vfill \mbox{}
\end{proof}
\pagebreak }{} }

\newcommand{\extra}{\ifthenelse{\boolean{skript}} {\newpage}{}}

\DeclareMathOperator{\Var}{Var} \DeclareMathOperator{\Cov}{Cov}
\DeclareMathOperator{\E}{{\mathbb E}}

\DeclareMathOperator{\N}{{\mathbb N}}

\DeclareMathOperator{\No}{N}

\newcommand{\bemerke}[1]{\ifthenelse{\boolean{skript}} {\marginpar{\scriptsize #1}}{}}

\newtheorem{satz}{Satz}[section]

\newtheorem{remark}[satz]{Remark}

\newtheorem{theorem}[satz]{Theorem}
\newtheorem{proposition}[satz]{Proposition}
\newtheorem{corollary}[satz]{Corollary}
\newtheorem{lemma}[satz]{Lemma}

\begin{document}
\global\long\def\I{\mathbf{1}}
\global\long\def\L{\mathcal{L}}
\global\long\def\mc{\omega(\Delta)}
\global\long\def\RG{\mathcal{R}}
\global\long\def\Eb{\mathbb{E}_{\sigma,b}}

\title[Inference for fractional Ornstein-Uhlenbeck type processes]{Inference for fractional Ornstein-Uhlenbeck type processes with periodic mean in the non-ergodic case}
%\date{Preliminary version, \today}
\author{Radomyra Shevchenko and Jeannette H.C. Woerner }
\address{Fakult\"at Mathematik, Technische Universit\"at Dortmund,
          Vogelpothsweg 87,
          D-44221 Dortmund, Germany}
\email{radomyra.shevchenko@math.tu-dortmund.de, jeannette.woerner@math.tu-dortmund.de}
\maketitle
\begin{abstract}
In the paper we consider the problem of estimating parameters entering the drift of a fractional Ornstein-Uhlenbeck type process in the non-ergodic case, when the underlying stochastic integral is of Young type. We consider the sampling scheme that the process is observed continuously on $[0,T]$ and $T\to\infty$. For known Hurst parameter $H\in(0.5, 1)$, i.e. the long range dependent case, we construct a least-squares type estimator and establish strong consistency. Furthermore, we prove a second order limit theorem which provides asymptotic normality for the parameters of the periodic function with a rate depending on $H$ and a non-central Cauchy limit result for the mean reverting parameter with exponential rate.  For the special case that the periodicity parameter is the weight of a periodic function, which integrates to zero over the period, we can even improve the rate to $\sqrt{T}$.
\end{abstract}
\pagenumbering{arabic}
\section{Introduction}
Parameter estimation in fractional diffusions has been actively studied in recent years, especially for equations of the Ornstein-Uhlenbeck type, i.e. of the form
\[dX_t = \alpha X_t dt + dB^H_t,\]
where $B^H_t$ is a fractional Brownian motion (fBm for short), which is a centred Gaussian process with almost surely continuous paths defined via its covariance structure
\[ \mathbb E [B^H_t B^H_s]=\frac{1}{2} \left(t^{2H}+s^{2H}-|t-s|^{2H} \right) \]
for $H\in (0,\,1)$.

With no initial condition imposed, the equation has an ergodic solution for $\alpha <0$, which is why this case is often called ergodic, as opposed to the non-ergodic case $\alpha >0$.

Several approaches are known for the estimation of $\alpha$, among them the MLE approach in \cite{KLB} which uses the so called fundamental martingales related to the underlying fBm and a minimum $L_1$-norm estimation in \cite{Rao} based on the techniques from \cite{KPi}.

Another possibility for the estimation is offered by the least squares approach, for which (following a heuristic notation) the term $\int_0^n (\dot{X}_t+ \alpha X_t)^2 dt$ is minimised, leading to the estimator
\[\tilde{\alpha}_n:=\frac{\int_0^n X_t dX_t}{\int_0^n X_t^2 dt}.\]
For $H=\frac{1}{2}$ the process $B^H$ is the classical Brownian motion, which allows for It\={o} integration in this definition. However, for $H\neq \frac{1}{2}$ fBm is not a semimartingale, so in order to define such an estimator one has to find a different (and suitable) kind of a stochastic integral, possible choices including pathwise (or Young type) and Skorokhod (or divergence type) integrals. Form the practical point of view pathwise integrals are preferred, however, for $\alpha<0$ this choice does not yield a consistent estimator. This was shown in \cite{HNu} alongside with consistency and asymptotic normality for the estimators defined with It\=o integrals for $H=\frac{1}{2}$ and for those defined with Skorokhod integrals for $H\in \Big(\frac{1}{2},\,\frac{3}{4}\Big)$. The non-ergodic case was treated in \cite{BESO} for $H\in \Big(\frac{1}{2},\,1\Big)$, where the same estimator defined with Young-type integrals is shown to be consistent and asymptotically Cauchy distributed. Note that for $H>1/2$ the processes possess long range dependence, which offers interesting possibilities for modelling.

In this paper we study Ornstein-Uhlenbeck type equations with an additional periodic mean term, i.e. equations of the form
\[dX_t = (L(t)+\alpha X_t) dt + dB^H_t\]
with a periodic, parametric function $L$, which can be used for modelling seasonalities, and assume continuous observations. In this case a consistent and asymptotically normal least squares-type estimator (coinciding with the MLE) for the drift parameters (including $\alpha <0$) was constructed for $H=\frac{1}{2}$ in \cite{DFK}, and it was later shown in \cite{DFW} (again, for $\alpha <0$) that a similarly defined estimator with divergence integrals has the properties of weak consistency and asymptotic normality for $H\in \Big(\frac{1}{2},\,\frac{3}{4}\Big)$. Strong consistency was then proved in \cite{BESV}.

 We will consider the same construction for the non-ergodic case and $H\in \Big(\frac{1}{2},\,1\Big)$ and investigate its asymptotic properties. We will prove strong consistency for our proposed estimator. Furthermore, we prove a second order limit theorem which provides asymptotic normality for the parameters of the periodic functions and a non-central Cauchy limiting result for the mean reverting parameter. Both limits are uncorrelated. This means in particular, we will show that the asymptotics is partly inherited from the ergodic case treated in \cite{DFW} and partly follows the results for the non-ergodic case in \cite{BESO}. In addition we show that in the special case that the periodicity parameter is the weight of a periodic function, whose integral over the period is zero, we get a faster rate of convergence independent of $H$, namely $\sqrt{T}$. In this case the component is uncorrelated to all other components. Note, that this case was not yet treated both in the ergodic and non-ergodic case. 

The structure of the paper is as follows. In Section \ref{sec:Setting} the setting will be explained in more detail, such that we can proceed with the motivation and definition of the estimator in Section \ref{sec:Constr}. Section \ref{sec:Aux} contains auxiliary convergence statements, and finally in Section \ref{sec:Asympt} the main results and their proofs are presented.
The Appendix provides some further technical results.

\section{Setting}\label{sec:Setting}
Let $(B^H_t)_{t\in\mathbb R^+}$ be a fractional Brownian motion with known Hurst index $H\in \big(\frac{1}{2},\,1\big)$. Consider a stochastic differential equation ( SDE) of the following form:\
\begin{equation}\label{eq:1}
\begin{split} 
X_t & = X_0 + \int_0^t L(s)+\alpha X_s ds + \int_0^t \sigma dB^H_s,\\
X_0 &= x_0 \in \mathbb R .
\end{split}
\end{equation}

$L$ is assumed to be a bounded $1$-periodic function which can be written as a linear combination of $p$ known bounded $1$-periodic $L^2([0,\,1])$-orthonormal functions with unknown real coefficients, $\mu=(\mu_1, \cdots, \mu_p)$ i.e.
\[L(s)=\sum_{i=1}^{p} \mu_i \phi_i (s)\text{ for all }s\in[0,\,1].\]
We assume that the mean-reverting parameter $\alpha>0$ is also unknown. As argued in \cite{Rao}, $\sigma$ can be estimated with probability one on any finite time interval, therefore it can be assumed to be known.

Moreover, it is important to define the stochastic integral appearing in the equation. In this paper we will consider it to be defined in Young's sense (cf. \cite{You}). The integrals are well defined due to H\"older continuity  of paths of the fractional Brownian motion of order $H$. Note that for deterministic integrands stochastic integrals in Young's sense almost surely coincide with Skorokhod integrals.

The equation \eqref{eq:1} has a solution with almost surely continuous paths, which can be written as
\[ X_t = e^{\alpha t}x_0 +e^{\alpha t}\int_0^t e^{-\alpha s} L(s) ds + \sigma e^{\alpha t}\int_0^t e^{-\alpha s}dB^H_s\]
for $\alpha >0$.
Let us fix the notation $\xi_t:=  e^{\alpha t}\int_0^t e^{-\alpha s}dB^H_s$, $\tilde{\xi}_t:= e^{-\alpha t}X_t$ as well as
\[\xi_{\infty}:= \int_0^\infty e^{-\alpha s}dB^H_s\]
and
\[\tilde{\xi}_{\infty}:=x_0 + \int_0^\infty  e^{-\alpha s} L(s) ds + \sigma\int_0^\infty e^{-\alpha s}dB^H_s.\]
We assume to observe $X$ continuously on [0,T] and derive limits for $T\to\infty$. 

\section{Construction of the estimator}\label{sec:Constr}
The estimator that we will consider has the same structure as the estimator defined in \cite{DFW} for the ergodic case. We will briefly outline the motivation as it was given there. The construction follows the least squares method applied to a discretised version of a more general equation
\[dX_t = \langle\theta,\,f(t,\,X_t)\rangle dt + \sigma dB^H_t,\]
where $\theta:=(\theta_1,\dots ,\theta_{p+1})$ is a parameter vector and $f(t,\,x):=(f_1(t,\,x),\dots ,f_{p+1}(t,\,x) )$ is a collection of known real-valued functions. For a time interval $[0,\,T]$ and a uniform mesh size $\Delta t:=T\slash N$ the least squares approach for the equations
\[ X_{(i+1)\Delta t}-X_{i\Delta t}=\sum_{j=1}^{p+1}f_j (i\Delta t,\, X_{i\Delta t})\theta_j \Delta t + \sigma (B^H_{(i+1)\Delta t}-B^H_{i\Delta t}),\,i\in \{1,\dots ,N \}, \]
yields the estimator
$\tilde{\theta}_{T,\,\Delta t}=Q_{T,\,\Delta t}^{-1}P_{T,\,\Delta t}$ with
\[Q_{T,\,\Delta t}=\left( \sum_{i=0}^N f_j(i\Delta t,\, X_{i\Delta t})f_k (i\Delta t,\,X_{i\Delta t})\Delta t\right)_{j,\,k\in \{1,\dots , p+1 \}}\]
and
\begin{align*}
P_{T,\,\Delta t}=\Big(\sum_{i=0}^N f_1(i\Delta t,\, X_{i\Delta t})&(X_{(i+1)\Delta t} - X_{i\Delta t}), \dots , \\
& \sum_{i=0}^N f_{p+1}(i\Delta t,\, X_{i\Delta t})(X_{(i+1)\Delta t} - X_{i\Delta t}) \Big)^T.
\end{align*}
Replacing the sums by their continuous counterparts and considering the special case $\theta = \vartheta := (\mu_1,\dots ,\mu_p,\,\alpha)$, $f(t,\,x)=(\phi_1 (t),\dots ,\phi_p(t),\, x)$ as well as putting $T=n$ we obtain the estimator $\hat{\vartheta}:=Q_n^{-1}P_n$ with
\[P_n=\left(\int_0^{n} \phi_1 (t) dX_t,\dots ,\int_0^{n} \phi_p (t) dX_t,\, \int_0^{n} X_t dX_t\right)\]
and
\[Q_n=\begin{pmatrix}
n E_p & a_n\\
a_n^T & b_n
\end{pmatrix},\]
where
\[a_n^T=\left(\int_0^{n} \phi_1 (t) X_t dt\dots , \int_0^{n} \phi_p (t) X_t dt\right),\]
\[b_n=\int_0^{n} X_t^2 dt.\]
The two following results are an immediate analogy to the calculations in \cite{DFW}.
\begin{proposition}
We have $\hat{\vartheta}_n = \vartheta + \sigma Q_n^{-1}R_n$, where
\[R_n = \left(\int_0^n \phi _1 (t) dB^H_t,\dots , \int_0^n \phi_p (t)dB^H_t,\, \int_0^n X_t dB^H_t\right)^T.\]
\end{proposition}
\begin{proof}
Since
\[\int_0^n \phi_i(t)dX_t= \sum_{j=1}^p \mu_j \int_0^n \phi_i(t)\phi_j(t)dt+\alpha \int_0^n \phi_i(t)X_t dt + \sigma \int_0^n \phi_i (t)dB^H_t\]
for $i\in \{1,\dots p \}$ and
\[\int_0^n X_t dX_t= \sum_{j=1}^p \mu_j \int_0^n X_t\phi_j(t)dt+\alpha \int_0^n X^2_t dt + \sigma \int_0^n X_t dB^H_t,\]
we have $P_n = Q_n\vartheta + \sigma R_n$, and the claim follows.
\end{proof}
\begin{proposition}
We have an explicit representation for $Q_n^{-1}$, namely
\[Q_n^{-1}=\frac{1}{n}\begin{pmatrix}
E_p + \gamma_n \Lambda_n \Lambda_n^t & -\gamma_n\Lambda_n\\
-\gamma_n \Lambda_n^t & \gamma_n
\end{pmatrix}\]
with
\[\Lambda_n = (\Lambda_{n,\,1},\dots , \Lambda_{n,\,p})^t =\left(\frac{1}{n}\int_0^n \phi_1 (t)X_t dt,\dots , \frac{1}{n}\int_0^n \phi_p (t)X_t dt\right)\]
and $\gamma_n = D_n^{-1} = \bigg(\frac{1}{n}\int_0^n X_t^2 dt-\sum_{i=1}^p \Lambda_{n,\,i}^2\bigg)^{-1}$.
\end{proposition}
\begin{proof}
This is a consequence of the fact that
\[\begin{pmatrix}
n E_p & -a_n\\
-a_n^T & b_n
\end{pmatrix}^{-1}= \frac{1}{n}\begin{pmatrix}
E_p + \gamma_n \Lambda_n \Lambda_n^t & \gamma_n\Lambda_n\\
\gamma_n \Lambda_n^t & \gamma_n
\end{pmatrix}, \]
which was proved in \cite{DFW}.
\end{proof}

\section{Auxiliary convergence results}\label{sec:Aux}
In this section we provide some convergence results for the different components of the estimators defined in the previous section. They help to identify the dominating terms for the asymptotic behaviour of the estimator.
The first lemma in this section as well as its proof are motivated by analogous results in \cite{BESO}.

\begin{lemma}
\label{l:1}
With the above notation we have $e^{-\alpha t}X_t \to \tilde{\xi}_{\infty}$ as well as $e^{-2\alpha t}\int_0^t X_s^2 ds \to \frac{\tilde{\xi}_{\infty}^2}{2\alpha}$ almost surely.
\end{lemma}

\begin{proof}
The first statement follows directly from the fact that $\xi_t\to\xi_\infty$ a.s. (shown in Lemma 2, \cite{BESO}):
\[e^{-\alpha t}X_t= x_0 +\int_0^t e^{-\alpha s} L(s) ds + \sigma \xi_t\to x_0 + \int_0^\infty  e^{-\alpha s} L(s) ds + \sigma \xi_\infty \text{ a.s.}\]
For the second statement we start by noticing that $\tilde{\xi}_t$ is a process with a.s. continuous paths. We have for each $t\in\mathbb R^+$:
\[\int_0^t X_s^2 ds \geq \int_{t\slash 2}^t e^{2\alpha s}\tilde{\xi}_s^2 ds \geq \frac{t}{2}e^{\alpha t}\inf_{\frac{t}{2}\leq s\leq t}\tilde{\xi}_s^2.\]
Since $\tilde{\xi}_t \to \tilde{\xi}_{\infty}$ a.s., it follows that
\[\lim_{t\to\infty } \inf_{\frac{t}{2}\leq s\leq t}\tilde{\xi}_s^2 = \tilde{\xi}_{\infty}^2 \text{ a.s.}\]
From the fact that $\xi_\infty \sim N(0,\, \frac{H\Gamma (2H)}{\alpha^{2H}})$ (shown in \cite{BESO}) we can conclude that $\tilde{\xi}_{\infty}$ also follows a (non-degenerate) normal distribution, and hence, $\lim_{t\to\infty} \int_0^t X_s^2 ds = \infty$ a.s. Therefore, we get by l'H\^opital's rule
\[\lim_{t\to\infty} \frac{ \int_0^t  e^{2\alpha s}\tilde{\xi}_s^2 ds}{e^{2\alpha t}}=\lim_{t\to\infty}\frac{\tilde{\xi}_t^2}{2\alpha}=\frac{\tilde{\xi}_{\infty}^2}{2\alpha}.\]
\end{proof}

\begin{lemma}
\label{l:2}
For $i\in \{1,\dots , p \}$ the following statements hold almost surely:
\begin{enumerate}
\item[(1)] $\frac{1}{n}\int_0^n \phi_i(t)dB^H_t\to 0$,
\item[(2)] $e^{-\alpha n}\Lambda_{ni}\sqrt{n}\to 0$,
\item[(3)] $nD_n e^{-2\alpha n}\to \frac{\tilde{\xi}_{\infty}^2}{2\alpha}$,
\item[(4)] $e^{-\alpha n}\frac{1}{\sqrt{n}}\int_0^n X_t dB^H_t\to 0$.
\end{enumerate}
\end{lemma}

\begin{proof}
\begin{enumerate}
\item[(1)] This is an application of Lemma \ref{l:4}: We have
\begin{align*}
\E [(\frac{1}{n}\int_0^n \phi_i(t)dB^H_t)^2]=\frac{1}{n^2}\int_0^n \int_0^n \phi_i (u)\phi_i (v)|u-v|^{2H-2}du dv\lesssim n^{2H-2},
\end{align*}
and the result follows for $k=2$.
\item[(2)]We write $\Lambda_{ni}$ as a sum of a deterministic and of a centred Gaussian part and show convergence separately:
\begin{align*}
e^{-\alpha n}\Lambda_{ni}=\frac{1}{n}e^{-\alpha n}\int_0^n \phi_i (t)&(e^{\alpha t}x_0 +e^{\alpha t}\int_0^t e^{-\alpha s}L(s)ds)dt\\
&+\frac{1}{n}e^{-\alpha n}\int_0^n \phi_i(t)\sigma e^{\alpha t}\xi_t dt=:A+B,
\end{align*}
where $\xi_t=\int_0^t e^{-\alpha r}dB^H_r$.
For the deterministic part we write
\begin{align*}
\sqrt{n}A = \frac{1}{\sqrt{n}}e^{-\alpha n}\int_0^n \phi_i (t)e^{\alpha t}x_0 dt &+ \frac{1}{\sqrt{n}}e^{-\alpha n}\int_0^n e^{\alpha t}\int_0^t e^{-\alpha s}L(s)dsdt\\
&=:A_1+A_2,
\end{align*}
and we can bound the two summands as follows:
\begin{align*}
|A_1|\lesssim \frac{1}{\sqrt{n}}e^{-\alpha n}\int_0^n e^{\alpha t}dt=\frac{1}{\sqrt{n}}-\frac{1}{\sqrt{n}}e^{-\alpha n}\to 0
\end{align*}
as well as
\begin{align*}
|A_2|\lesssim \frac{1}{\sqrt{n}}e^{-\alpha n}\int_0^n e^{\alpha t}\int_0^t e^{-\alpha s}ds dt = \frac{1}{\sqrt{n}}e^{-\alpha n}\int_0^n e^{\alpha t}dt - \frac{1}{\sqrt{n}}e^{-\alpha n}\to 0.
\end{align*}
We have shown convergence for the deterministic part and now we will calculate the second moment of the Gaussian part in order to apply Lemma \ref{l:4}.
\begin{align*}
\E [(\sqrt{n}B)^2]=&\E [(\frac{1}{\sqrt{n}}e^{-\alpha n}\int_0^n \phi_i(t)\sigma e^{\alpha t}\xi_t dt)^2]\\
=&\frac{1}{n}e^{-2\alpha n}\int_0^n\int_0^n \phi_i(t)\phi_i (s)\sigma^2 e^{\alpha t}e^{\alpha s}\E [\xi_t\xi_s]ds dt
\end{align*}
and we get by treating the stochastic integrals as Skorokhod integrals
\begin{align*}
\E [\xi_t\xi_s] = \int_0^t \int_0^s e^{-\alpha r}e^{-\alpha v}|r-v|^{2H-2}dvdr.
\end{align*}
In total, we obtain
\begin{align*}
&\E [(\sqrt{n}B)^2]\\
&=\frac{1}{n}e^{-2\alpha n}\sigma^2\int_0^n\int_0^n \phi_i(t)\phi_i (s)  \int_0^t \int_0^s e^{\alpha s-\alpha r}e^{\alpha t-\alpha v}|r-v|^{2H-2}dvdrdsdt\\
&=\frac{1}{n}e^{-2\alpha n}\sigma^2\int_0^n\int_0^n |r-v|^{2H-2}\int_v^n \int_r^n \phi_i(t)\phi_i (s)  e^{\alpha s-\alpha r}e^{\alpha t-\alpha v} ds dt dv dr\\
&\lesssim \frac{1}{\alpha^2}\frac{1}{n}e^{-2\alpha n}\sigma^2\int_0^n\int_0^n |r-v|^{2H-2} (e^{\alpha n-\alpha v}-1)(e^{\alpha n-\alpha r}-1)drdv\\
&\simeq \frac{1}{n}\int_0^n\int_0^n |r-v|^{2H-2} (e^{-\alpha v}-e^{-\alpha n})(e^{-\alpha r}-e^{-\alpha n})drdv\\
&\leq \frac{1}{n}\int_0^n\int_0^n |r-v|^{2H-2} e^{-\alpha v}e^{-\alpha r}drdv\lesssim \frac{1}{n},
\end{align*}
because the last integral is bounded (this was shown in \cite{HNu}). Lemma \ref{l:4} yields almost sure convergence to zero and hence the desired result.
\item[(3)] This follows from the previous result and Lemma \ref{l:1}:
\begin{align*}
D_nne^{-2\alpha n}=e^{-2\alpha n}\int_0^n X_t^2 dt- \underbrace{\sum_{i=1}^p (\sqrt{n}\Lambda_{ni}e^{-\alpha n})^2}_{\to 0\text{ by }(2)}\to \frac{\tilde{\xi}_{\infty}^2}{2\alpha}.
\end{align*}
\item[(4)] We plug in the expression $X_t$ and get
\begin{align*}
e^{-\alpha n}&\frac{1}{\sqrt{n}}\int_0^n X_t dB^H_t=e^{-\alpha n}\frac{1}{\sqrt{n}}\int_0^n e^{\alpha t}x_0 dB^H_t\\
&+ e^{-\alpha n}\frac{1}{\sqrt{n}}\int_0^n e^{\alpha t} \int_0^t e^{-\alpha s}L(s)ds dB^H_t\\
&+ e^{-\alpha n}\frac{1}{\sqrt{n}}\int_0^n \sigma e^{\alpha t}\int_0^t e^{-\alpha s}dB^H_s dB^H_t=:A+B+C.
\end{align*}
The integral in $A$ can again be interpreted as a Skorokhod integral (yielding a centred Gaussian random variable) which allows us the computation of its $L^2$ norm:
\begin{align*}
\E [A^2] &= x_0^2\frac{1}{n}e^{-2\alpha n}\int_0^n\int_0^n e^{\alpha u}e^{\alpha v}|u-v|^{2H-2}du dv\\
&=x_0^2\frac{1}{n}\underbrace{\int_0^n\int_0^n e^{-\alpha (n-u)}e^{-\alpha (n-v)}|u-v|^{2H-2}du dv}_{=:I_n}\lesssim \frac{1}{n},
\end{align*}
because $I_n$ is bounded as was shown in \cite{HNu}. Lemma \ref{l:4} implies almost sure convergence.
For $B$, which is also a centred Gaussian sequence, the calculation is similar:
\begin{align*}
&\E [B^2]\\
&=\frac{1}{n}e^{-2\alpha n} \int_0^n\int_0^n e^{\alpha u} \int_0^u e^{-\alpha s}L(s)ds e^{\alpha v}\int_0^v e^{-\alpha r}L(r)dr|u-v|^{2H-2}du dv\\
&\lesssim \frac{1}{n}e^{-2\alpha n} \int_0^n\int_0^n e^{\alpha u} (1-e^{-\alpha u}) e^{\alpha v}(1-e^{-\alpha v})|u-v|^{2H-2}du dv\\
&=\frac{1}{n}e^{-2\alpha n} \int_0^n\int_0^n (e^{\alpha u}-1)(e^{\alpha v}-1)|u-v|^{2H-2}du dv\leq \frac{1}{n}I_n\lesssim \frac{1}{n},
\end{align*}
and the almost sure convergence follows.
For $C$ we use Lemma 4 from \cite{BESO} to decompose the double integral:
\begin{align*}
C=e^{-\alpha n}&\frac{1}{\sqrt{n}}\sigma (\int_0^n  e^{\alpha s}dB^H_s \int_0^t e^{-\alpha r}dB^H_r - \int_0^n  e^{-\alpha s}\int_0^s e^{\alpha r}\delta B^H_r \delta B^H_s\\
& - H(2H-1)\int_0^n  e^{-\alpha s}\int_0^s e^{\alpha r} |s-r|^{2H-2}dr ds)=:C_1-C_2-C_3,
\end{align*}
where $\delta$ stands for the Skorokhod integral. We show almost sure convergence for the three summands:
\begin{align*}
C_1=\sigma e^{-\alpha n}\frac{1}{\sqrt{n}}\int_0^n  e^{\alpha s}dB^H_s \xi_t,
\end{align*}
and hence we know from \cite{BESO} that $\xi_t\to\xi_{\infty}\sim N(0,\,\frac{H\Gamma (2H)}{\alpha^{2H}})$ a.s., it is enough to show that $e^{-\alpha n}\frac{1}{\sqrt{n}}\int_0^n  e^{\alpha s}dB^H_s \to 0$ almost surely. Since it is a centred Gaussian sequence, we again rely on Lemma \ref{l:4} and compute the respective variances:
\begin{align*}
\E [(e^{-\alpha n}\frac{1}{\sqrt{n}}&\int_0^n  e^{\alpha s}dB^H_s)^2]\\
&\simeq e^{-2\alpha n}\frac{1}{n}\int_0^n\int_0^n e^{\alpha s}e^{\alpha r}|s-r|^{2H-2}dsdr = \frac{1}{n}I_n\lesssim \frac{1}{n}.
\end{align*}
In order to treat $C_2$ note that by Lemma 7 in \cite{BESO}
\[ Y_n:= e^{-\frac{\alpha n}{2}}\int_0^n  e^{-\alpha s}\int_0^s e^{\alpha r}\delta B^H_r \delta B^H_s\stackrel{L^2}{\to} 0,\]
and consequently $\E [Y_n^2]$ is bounded. Since, moreover, $Y_n$ is centred (as it is a Skorokhod integral), Markov inequality helps achieve the summability of tails:
\begin{align*}
\sum_{n=1}^{\infty} P(|C_2(n)|\geq \varepsilon)= \sum_{n=1}^{\infty} & P(|\frac{1}{\sqrt{n}}e^{-\frac{\alpha n}{2}} Y_n|\geq \varepsilon)\\
&\leq \sum_{n=1}^{\infty} \frac{\E [Y_n^2]}{\varepsilon^2 n e^{\frac{\alpha n}{2}} }\lesssim \sum_{n=1}^{\infty} \frac{1}{n e^{\frac{\alpha n}{2}} } <\infty,
\end{align*}
and almost sure convergence to zero follows.
Finally, Lemma 7 in \cite{BESO} ensures that $C_3 e^{\frac{\alpha n}{2}}\sqrt{n}$ converges to zero, which implies that also $C_3$ itself goes to zero as $n$ tends to infinity. This completes the proof  of the initial claim.
\end{enumerate}
\end{proof}

\begin{corollary} \label{cor:1}
For $\beta <\frac{1}{2}$ also $n^{\beta} e^{-\alpha n}\Lambda_{ni}\sqrt{n}\to 0 $ as well as  $n^{\beta} e^{-\alpha n}\frac{1}{\sqrt{n}}\int_0^n X_t dB^H_t\to 0$ almost surely.
\end{corollary}
\begin{proof}
The deterministic part of the sequence $n^{\beta} e^{-\alpha n}\Lambda_{ni}\sqrt{n}$ (i.e. $n^{\beta}\sqrt{n} A$, cf. the notation from the proof of (2) in \ref{l:2}) is bounded up to a constant by $n^{\beta-0.5}$ and the variance of the random part by $n^{2\beta -1}$. This yields polynomial rates of convergence, thus Lemma \ref{l:4} still can be applied and we obtain almost sure convergence.
The same argument holds for the second convergence result. Lemma \ref{l:4} can still be applied for $A$, $B$ and $C_1$ (from the proof of (4) in \ref{l:2}), and for $C_2$ and $C_3$ the additional factor $n^\beta$ changes nothing in the structure of the arguments, so the proofs can be followed verbatim.
\end{proof}
\section{Asymptotic properties of $\hat{\vartheta}$}\label{sec:Asympt}
In this section we examine the asymptotic properties of our estimator. First we prove strong consistency, than we provide a second order limit result, which shows the substantially different behaviour of the parameters of the periodic function and the mean reverting parameter, namely we get both different limiting distributions and different rates. Finally, we will show that for basis functions $\phi$ with $\int_0^1 \phi(s) ds=0$ the rate in the central limit theorem improves to $\sqrt{n}$ independent of $H$. In this case we provide two representations of the asymptotic variance, one involving sums over the Riemann zeta function, the other an integral representation.
\begin{theorem}\label{th:1}
$\hat{\vartheta}$ is strongly consistent, i.e.
\begin{enumerate}
\item [(1)] for $i\in \{1,\dots , p\}$
\begin{align*}
\hat{\mu_i}-\mu_i & = \sigma \frac{1}{n}(\int_0^n \phi_i (t)dB^H_t\\
&+ \frac{1}{D_n}\sum_{j=1}^p \Lambda_{ni}\Lambda_{nj}\int_0^n \phi_j (t)dB^H_t - \frac{1}{D_n}\Lambda_{ni}\int_0^n X_t dB^H_t)\to 0,
\end{align*}
\item[(2)] $\hat{\alpha}-\alpha= -\sigma \frac{1}{nD_n}(\sum_{i=1}^p \Lambda_{ni}\int_0^n \phi_i (t)dB^H_t - \int_0^n X_t dB^H_t)\to 0$,
\end{enumerate}
both almost surely.
\end{theorem}
\begin{proof}
We treat each summand separately and exploit Lemma \ref{l:2}.
\begin{enumerate}
\item[(1)] Let us denote $M_1:= \frac{1}{n}\int_0^n \phi_i (t)dB^H_t$, $M_{2j}:= \frac{1}{n}\frac{1}{D_n}\Lambda_{ni}\Lambda_{nj}\int_0^n \phi_j (t)dB^H_t$, $M_3:=  \frac{1}{n}\frac{1}{D_n}\Lambda_{ni}\int_0^n X_t dB^H_t $. In order to prove the claim we have to show that each of these summands converges to zero almost surely.
For $M_1$ this was shown in Lemma \ref{l:2} (1). To see this for $M_{2j}$ we rewrite it as follows:
\begin{align*}
M_{2j}&=\frac{1}{n}\frac{1}{D_n}\Lambda_{ni}\Lambda_{nj}\int_0^n \phi_j (t)dB^H_t\\
& = \underbrace{\frac{1}{n D_ne^{-2\alpha n}}}_{\to \frac{2\alpha}{\tilde{\xi}_{\infty}^2} \text{ by }\ref{l:2} (3)} \underbrace{(e^{-\alpha n}\Lambda_{ni}\sqrt{n})}_{\to 0 \text{ by }\ref{l:2} (2)}\underbrace{(e^{-\alpha n}\Lambda_{nj}\sqrt{n})}_{\to 0 \text{ by }\ref{l:2} (2)} \underbrace{\frac{1}{n}\int_0^n \phi_j (t)dB^H_t}_{\to 0 \text{ by }\ref{l:2} (1)},
\end{align*}
and since $\tilde{\xi}_{\infty}$ is almost surely nonzero, the whole expression converges a.s. to zero.
$M_3$ can also be rewritten in a way that makes the convergence statement obvious:
\begin{align*}
M_3 &= \frac{1}{n}\frac{1}{D_n}\Lambda_{ni}\int_0^n X_t dB^H_t\\
& = \underbrace{\frac{1}{n D_ne^{-2\alpha n}}}_{\to \frac{2\alpha}{\tilde{\xi}_{\infty}^2} \text{ by }\ref{l:2} (3)}  \underbrace{(e^{-\alpha n}\Lambda_{ni}\sqrt{n})}_{\to 0 \text{ by }\ref{l:2} (2)}  \underbrace{(e^{-\alpha n}\frac{1}{\sqrt{n}}\int_0^n X_t dB^H_t)}_{\to 0 \text{ by }\ref{l:2} (4)},
\end{align*}
the claim follows with the same argument as above and completes the proof of the theorem's statement.
\item[(2)] In this case we also start by introducing a notation for each type of summands. Let us denote $A_{1i}:=\frac{1}{nD_n} \Lambda_{ni}\int_0^n \phi_i (t)dB^H_t$ and $A_2:= \frac{1}{nD_n}\int_0^n X_t dB^H_t$. For the first type of summands we write
\begin{align*}
A_{1i}&= \frac{1}{nD_n} \Lambda_{ni}\int_0^n \phi_i (t)dB^H_t\\
&= \underbrace{\frac{1}{n D_ne^{-2\alpha n}}}_{\to \frac{2\alpha}{\tilde{\xi}_{\infty}^2} \text{ by }\ref{l:2} (3)}  \underbrace{(e^{-\alpha n}\Lambda_{ni}\sqrt{n})}_{\to 0 \text{ by }\ref{l:2} (2)} \underbrace{\sqrt{n}e^{-\alpha n}}_{\to 0} \underbrace{\frac{1}{n}\int_0^n \phi_i (t)dB^H_t}_{\to 0 \text{ by }\ref{l:2} (1)}
\end{align*}
and for the second kind we obtain
\begin{align*}
A_2 &=  \frac{1}{nD_n}\int_0^n X_t dB^H_t\\
&= \underbrace{\frac{1}{n D_ne^{-2\alpha n}}}_{\to \frac{2\alpha}{\tilde{\xi}_{\infty}^2} \text{ by }\ref{l:2} (3)}  \underbrace{\sqrt{n}e^{-\alpha n}}_{\to 0} \underbrace{(e^{-\alpha n}\frac{1}{\sqrt{n}}\int_0^n X_t dB^H_t)}_{\to 0 \text{ by }\ref{l:2} (4)}.
\end{align*}
Both calculations yield almost sure convergence of the summands (again, using the argument given in (1)) and thus provide the proof for the initial claim.
\end{enumerate}
\end{proof}
The next lemma is an auxiliary result for the second order limit theorem that will be proved later.
\begin{lemma}\label{l:3}
Let $F$ be any $\sigma (B^H)$-measurable random variable such that $P(F<\infty)=1$. Then, as $n\to\infty$,
\[ (n^{-H}\delta_n (\phi_1),\dots , n^{-H}\delta_n(\phi_p),\, F,\, e^{-\alpha n}\delta_n( e^{\alpha \cdotp}))\stackrel{d}{\to}(Z_1,\dots , Z_p,\, F,\,Z), \]
where $\delta_n$ is the integral over $[0,\,n]$ with respect to $B^H$, $Z_1,\dots , Z_p$ are centred and jointly normally distributed with the covariance matrix $(\int_0^1 \phi_i(x)dx\int_0^1 \phi_j(x)dx)_{i,\,j=1,\dots , p}$ and $((Z_1,\dots , Z_p),\, F,\, Z)$ are independent. Moreover, $\Var (Z)=\frac{H\Gamma (2H)}{\alpha^{2H}}$.
\end{lemma}
\begin{proof}
Due to an approximation argument rigorously explained in \cite{ESN} it is enough to show that for any $d\geq 1$, $s_1,\dots , s_d\in [0,\,\infty)$
\begin{align*}
(n^{-H}\delta_n (\phi_1),\dots , n^{-H}\delta_n(\phi_p), &\, B^H_{s_1},\dots , B^H_{s_d} ,\, e^{-\alpha n}\delta_n( e^{\alpha \cdotp}))\\
& \stackrel{d}{\to}(Z_1,\dots , Z_p,\,  B^H_{s_1},\dots , B^H_{s_d},\,Z)
\end{align*}
as $n\to\infty$. The left hand side is a Gaussian vector, and hence it suffices to determine the limits of the covariances. It was shown in \cite{BESO} that the limits of $\Cov (B^H_s,\,   e^{-\alpha n}\delta_n( e^{\alpha \cdotp}))$ and $\Var ( e^{-\alpha n}\delta_n( e^{\alpha \cdotp}))$ are as claimed. Moreover, in \cite{BESV} the joint limiting distribution of $(n^{-H}\delta_n (\phi_1),\dots , n^{-H}\delta_n(\phi_p))$ was established. Therefore, we only have to show that $\Cov ( n^{-H}\delta_n (\phi_i), \, B^H_s )$ and $\Cov ( n^{-H}\delta_n (\phi_i),\,  e^{-\alpha n}\delta_n( e^{\alpha \cdotp}) )$ converge to zero.
For the first statement recall that $B^H_s = \int_0^n 1_{[0,\, s]}dB^H_t$ for any $n\geq s$. Then we can write (for $n$ large enough) due to the isometry property of the integrals:
\begin{align*}
\E &[ n^{-H}\delta_n (\phi_i) B^H_s] \lesssim n^{-H}\int_0^n \int_0^s |u-v|^{2H-2}du dv\\
& =n^{-H}\int_0^s \int_{-v}^{n-v}|z|^{2H-2}dz dv = n^{-H}\int_0^s \int_0^v z^{2H-2}dz + \int_0^{n-v}z^{2H-2} dz dv\\
& = \underbrace{n^{-H}\int_0^s v^{2H-1}dv}_{\to 0}+ n^{-H}\int_0^s (n-v)^{2H-1}dv\lesssim n^{-H}\int_{n-s}^n z^{2H-1}dz\\
& =n^{-H}(n^{2H}-(n-s)^{2H})\stackrel{\text{binom. series}}{=}n^{-H}O(n^{2H-1})=O(n^{H-1}),
\end{align*}
which tends to zero as $n$ tends to infinity.

For the second convergence refer to Proposition \ref{prop:2} in the Appendix for the estimation $\int_0^t e^{\alpha u}u^{2H-2}du \lesssim t^{2H-2}e^{\alpha t}$. We use this for our calculation:
\begin{align*}
\E &[ n^{-H}\delta_n (\phi_i) e^{-\alpha n}\delta_n( e^{\alpha \cdotp}) ] \lesssim n^{-H} e^{-\alpha n} \int_0^n \int_0^n e^{\alpha v} |u-v|^{2H-2}du dv\\
& = n^{-H} e^{-\alpha n} \int_0^n e^{\alpha u} \int_0^n e^{\alpha (v-u)}|v-u|^{2H-2}dv du\\
& =n^{-H}e^{-\alpha n} \int_0^n \left(\underbrace{\int_0^u e^{-\alpha z} z^{2H-2}dz}_{\text{bdd}}+ \int_0^{n-u} e^{\alpha z}z^{2H-2}dz\right) du\\
& \lesssim n^{-H}e^{-\alpha n} \int_0^n e^{\alpha u} e^{\alpha (n-u)}(n-u)^{2H-2}du = n^{-H}n^{2H-1}\to 0. 
\end{align*}
\end{proof}
Now we can proceed with the second order limit theorem for our estimator.
\begin{theorem}\label{th:2}
\[(n^{1-H}(\hat{\mu}_1-\mu_1,\dots ,\hat{\mu}_p-\mu_p),\, e^{\alpha n}(\hat{\alpha}-\alpha))\stackrel{d}{\to}\sigma (Z_1,\dots ,\, Z_p,\, Z_{p+1})\]
with $Z_1,\dots ,Z_p$ as above and $Z_{p+1}= 2\alpha N\slash M$ with $N\sim \No (0,\,1)$ and
\[M\sim \No \left(\frac{ \alpha^H}{\sqrt{H\Gamma(2H)}} \left(x_0+\int_0^{\infty}e^{-\alpha s}L(s)ds\right),\,1\right)\]
independent of $N$. Moreover, $(Z_1, \dots ,Z_p)$ and $Z_{p+1}$ also are independent.
\end{theorem}
This result reflects the structure of the estimator: In the first $p$ components the additive term $\sigma \frac{1}{n}\int_0^n \phi_i(t)dB^H_t$ is the slowest summand (note that it does not include the solution process $X$ and is, therefore, not influenced by its exponential growth), which yields the same rates of convergence as in the ergodic case. The estimator for $\alpha$, however, does not contain such a term; it converges with the same exponential rate as the estimator in \cite{BESO}. The limiting distribution is also structurally similar to the case $L\equiv 0$. As mentioned in \cite{Moe}, if the estimator from \cite{BESO} is applied for an equation with a non zero starting value, the limiting distribution will also contain this value as an additional additive term in the denominator.
Moreover, due to the possibility of considering Young integrals and exploiting different techniques in the proofs our results are valid for $H\in \Big(\frac{1}{2},\,1\Big)$ in contrast to only $H\in \Big(\frac{1}{2},\,\frac{3}{4}\Big)$ for the ergodic case in \cite{DFW}.
\begin{proof}
First of all we divide the error into parts that contribute to the limit and the rest. We use the notation from the previous theorem and write: $n^{1-H}(\hat{\mu}_1-\mu_1)=\sigma (n^{1-H}M_1+n^{1-H}(\sum_{j=1}^p M_{2j}+M_3))$, $e^{\alpha n}(\hat{\alpha}-\alpha)=\sigma (-e^{\alpha n}\sum_{j=1}^p A_{1j}+ e^{\alpha n} A_2)$.
Now we will identify the rest terms by showing: $n^{1-H}(\sum_{j=1}^p M_{2j}+M_3)$ and $e^{\alpha n}\sum_{j=1}^p A_{1j}$ converge to zero almost surely. For $M_{2j}$ and $M_3$ this follows from the fact that they contain the factor $(e^{-\alpha n}\Lambda_{nj}\sqrt{n})$ which would still converge to zero if multiplied by $n^{1-H}$, since $1-H<0.5$.
$A_{1j}$ contains the factor $(e^{-\alpha n}\Lambda_{ni}\sqrt{n}) \sqrt{n}e^{-\alpha n} \frac{1}{n}\int_0^n \phi_j (t)dB^H_t$ converging  to zero almost surely. The remainder $\frac{1}{n D_ne^{-2\alpha n}}$ tends almost surely to a random variable. We write
\begin{align*}
e^{\alpha n}&(e^{-\alpha n}\Lambda_{ni}\sqrt{n}) \sqrt{n}e^{-\alpha n} \frac{1}{n}\int_0^n \phi_j (t)dB^H_t = (e^{-\alpha n}\Lambda_{ni}\sqrt{n}) \sqrt{n}\frac{1}{n}\int_0^n \phi_j (t)dB^H_t\\
& = (e^{-\alpha n}\Lambda_{ni}\sqrt{n} n^{H-0.5})\left(n^{1-H} \frac{1}{n}\int_0^n \phi_j (t)dB^H_t\right).
\end{align*}
The factor $e^{-\alpha n}\Lambda_{ni}\sqrt{n} n^{H-0.5}$ converges to zero almost surely, because $H-0.5<0.5$ and the factor $n^{1-H} \frac{1}{n}\int_0^n \phi_j (t)dB^H_t$ converges in distribution to a normal random variable (this being a consequence of the previous lemma). In total we conclude that the above expression converges to zero in distribution and therefore in probability. Thus, also the whole term $e^{\alpha n}A_{1j}$ converges to zero in probability.

The next step is to consider and rewrite $A_2$. For this we apply the change of variables formula for Young integrals to the functions $e^{-\alpha n}X_n$ and $\int_0^n e^{\alpha t}dB^H_t$. We obtain the following formula:
\begin{align*}
\int_0^n X_s dB^H_s = \int_0^n & e^{\alpha t} dB^H_t \tilde{\xi}_n -\int_0^n e^{-\alpha t}L(t)\int_0^t e^{\alpha s}dB^H_s dt\\
& - \int_0^n \sigma e^{-\alpha t}\int_0^t e^{\alpha s}dB^H_s dB^H_t=:S_1+S_2+S_3,
\end{align*}
with which we can substitute the term $\int_0^n X_s dB^H_s$ in $A_2$. We will now show that only $S_1$ contributes to the convergence statement. Since
\[e^{\alpha n}A_2 = \frac{1}{n D_ne^{-2\alpha n}} e^{-\alpha n} \int_0^n X_t dB^H_t\]
and the denominator converges almost surely, it is enough to show that $ e^{-\alpha n} (S_2+S_3)$ tend to zero in probability. For $S_3$ this has been shown in \cite{BESO}, so we only show this for $S_2$. As a Lebesgue integral of a Gaussian process $e^{-\alpha n}S_2$ is again centred Gaussian, therefore showing convergence of the second moments will suffice:
\begin{align*}
\E &\left[ \left( e^{-\alpha n} \int_0^n e^{-\alpha t}L(t)\int_0^t e^{\alpha s}dB^H_s dt \right)^2\right]\\
&\lesssim e^{-2 \alpha n}\int_0^n \int_0^n e^{-\alpha u} L(u)e^{-\alpha v} L(v)\int_0^u \int_0^v e^{\alpha s}e^{\alpha r}|s-r|^{2H-2}ds dr du dv\\
&\lesssim e^{-2 \alpha n}\int_0^n \int_0^n \int_0^u \int_0^v |s-r|^{2H-2}ds dr du dv \lesssim e^{-2 \alpha n} n^{2H+2}\to 0
\end{align*}
as $n$ tends to infinity.

For the last step of the proof we apply Lemma \ref{l:3} to $F=\tilde{\xi}_{\infty}$ and obtain
\[  (n^{-H}\delta_n (\phi_1),\dots , n^{-H}\delta_n(\phi_p),\, \tilde{\xi}_{\infty},\, e^{-\alpha n}\delta_n( e^{\alpha \cdotp}))\stackrel{d}{\to}(Z_1,\dots , Z_p,\, \tilde{\xi}_{\infty},\,Z),\]
and consequently
\[  \left(n^{-H}\delta_n (\phi_1),\dots , n^{-H}\delta_n(\phi_p),\, \frac{e^{-\alpha n}\int_0^n e^{\alpha t}dB^H_t}{\tilde{\xi}_{\infty}} \right)\stackrel{d}{\to}\left(Z_1,\dots , Z_p,\, \frac{Z}{\tilde{\xi}_{\infty}}\right),\]
where $Z\sim \sqrt{\frac{H\Gamma (2H)}{\alpha^{2H}}}\No (0,\, 1)$ and
\[\tilde{\xi}_{\infty}\sim \sqrt{\frac{H\Gamma (2H)}{\alpha^{2H}}}  \No \left(\frac{ \alpha^H}{\sqrt{H\Gamma(2H)}} \left(x_0+\int_0^{\infty}e^{-\alpha s}L(s)ds\right),\,1\right) .\]
Now note additionally that
\[(1,\dots, 1,\, \frac{\tilde{\xi}_n \tilde{\xi}_{\infty}}{n D_ne^{-2\alpha n}})\stackrel{\text{a.s.}}{\to} (1,\dots , 1,\, 2\alpha).\]
Multiplying both vectors elementwise using Slutsky's lemma yields
\[  (n^{-H}\delta_n (\phi_1),\dots , n^{-H}\delta_n(\phi_p),\, \frac{e^{-\alpha n}S_1}{n D_ne^{-2\alpha n}} )\stackrel{d}{\to}(Z_1,\dots , Z_p,\, 2\alpha \frac{Z}{\tilde{\xi}_{\infty}}),\]
which is all that we needed to show, since all the other summands converge to zero in probability. Note that we inherit the independence statement directly from Lemma \ref{l:3}.
\end{proof}

Consider the special case of a basis element $\phi_k$, $k\in \{1,\dots ,p\}$, which integrates to zero on $[0,\,1]$. The results of our theorems continue to hold, but the limiting vector $(Z_1,\dots ,Z_p)$ will have a zero entry at $Z_k$. This suggests that the convergence of the estimator's $k$th component might be of a better order than $n^{H-1}$. Indeed, one obtains the following facts.
\begin{proposition}\label{prop:1}
If $\phi_k$ for $k\in \{1,\dots ,p\}$ is such that $\int_0^1 \phi_k(t)dt=0$, then
\[\sqrt{n}(\hat{\mu}_k-\mu_k)\stackrel{d}{\to}\sigma H(2H-1)\bar{Z}_k,\]
where $\bar{Z}_k$ is a zero mean Gaussian random variable with variance 
\begin{align*}
\int_0^1 &\int_0^1 \phi_k (t)\phi_k (s)|t-s|^{2H-2}dt ds\\
&+ \sum_{l=1}^\infty 2 {{2H-2}\choose{2l}}\zeta (2l+2-2H) \int_0^1\int_0^1 \phi_k (t)\phi_k (s)(t-s)^{2l}dt ds,
\end{align*}
where $\zeta$ denotes the Riemann zeta function.
\end{proposition}
\begin{proof}
Recall that
\[\sqrt{n}(\hat{\mu}_1-\mu_1)=\sigma \left(\sqrt{n} M_1+\sqrt{n}\left(\sum_{j=1}^p M_{2j}+M_3\right)\right)\]
with the notation from Theorem \ref{th:1}. As in Theorem \ref{th:2}, Corollary \ref{cor:1} ensures that $\sqrt{n}M_{2j}$ and $\sqrt{n}M_3$ converge to zero almost surely. Given that
\[\sigma\sqrt{n}M_1=\sigma \frac{1}{\sqrt{n}}\int_0^n \phi_k(t)dB^H_t,\]
it is enough for our claim to investigate the term \(\mathbb E \left[ \left(\frac{1}{\sqrt{n}}\int_0^n \phi_k(t)dB^H_t\right)^2\right]\).\\
With $\alpha_H= H(2H-1)$ we have by isometry and periodicity:
\begin{align}\label{eq:2}\nonumber
\frac{1}{\alpha_H} & \mathbb E \left[ \left(\frac{1}{\sqrt{n}}\int_0^n \phi_k(t)dB^H_t\right)^2\right]\\ \nonumber
&=\frac{1}{n}\int_0^n \int_0^n \phi_k (t)\phi_k (s)|t-s|^{2H-2}dt ds\\ \nonumber
&=\frac{1}{n}\sum_{i,\,j =0}^{n-1}\int_0^1\int_0^1 \phi_k(t)\phi_k(s)|t+i-s-j|^{2H-2}dt ds\\ \nonumber
&= \frac{1}{n}\int_0^1 \int_0^1 \phi_k (t)\phi_k (s) n |t-s|^{2H-2}dtds\\ \nonumber
&\qquad +\frac{1}{n} \int_0^1\int_0^1 \phi_k(t)\phi_k(s)\sum_{i>j}|t-s+i-j|^{2H-2}dtds\\ 
&\qquad +\frac{1}{n} \int_0^1\int_0^1 \phi_k(t)\phi_k(s)\sum_{j>i}|s-t+j-i|^{2H-2}dtds.
\end{align}
The first summand is constant with respect to $n$, hence, it remains to consider the second and the third one (which are equal for symmetry reasons). By rearranging the sum in the second summand, we obtain the following:
\begin{align*}
\frac{1}{n} & \int_0^1\int_0^1 \phi_k(t)\phi_k(s)\sum_{i>j}|t-s+i-j|^{2H-2}dtds\\
&=\frac{1}{n}\int_0^1 \int_0^1 \phi_k(t)\phi_k(s) \sum_{m=1}^{n-1}(n-m)|t-s+m|^{2H-2}dtds\\
&=\frac{1}{n}\int_0^1 \int_0^1 \phi_k(t)\phi_k(s) \sum_{m=1}^{n-1}(n-m)m^{2H-2}\left(\frac{t-s}{m}+1\right)^{2H-2}dtds\\
&=\frac{1}{n}\int_0^1 \int_0^1 \phi_k(t)\phi_k(s) \sum_{m=1}^{n-1}nm^{2H-2}\left(\frac{t-s}{m}+1\right)^{2H-2}dtds\\
&\qquad - \frac{1}{n}\int_0^1 \int_0^1 \phi_k(t)\phi_k(s) \sum_{m=1}^{n-1}m\cdotp m^{2H-2}\left(\frac{t-s}{m}+1\right)^{2H-2}dtds.\\
\end{align*}
Now we can use the binomial series expansion to get
\[\left(\frac{t-s}{m}+1\right)^{2H-2} = \sum_{l=0}^\infty {2H-2 \choose l}(t-s)^lm^{-l}\]
and use the zero integral assumption in order to evaluate the above expression. We conclude:
\begin{align*}
\frac{1}{n}&\int_0^1 \int_0^1 \phi_k(t)\phi_k(s) \sum_{m=1}^{n-1}nm^{2H-2}\left(\frac{t-s}{m}+1\right)^{2H-2}dtds\\
&= \int_0^1 \int_0^1 \phi_k(t)\phi_k(s) \sum_{m=1}^{n-1}m^{2H-2} \sum_{l=2}^\infty {2H-2 \choose l}(t-s)^lm^{-l} dtds\\
&= \int_0^1 \int_0^1 \phi_k(t)\phi_k(s)  \sum_{l=2}^\infty {2H-2 \choose l}(t-s)^l \sum_{m=1}^{n-1}m^{2H-2-l} dtds.
\end{align*}
By dominated convergence we now obtain
\begin{align*}
\lim_{n\to\infty} &\int_0^1 \int_0^1 \phi_k(t)\phi_k(s)  \sum_{l=2}^\infty {2H-2 \choose l}(t-s)^l \sum_{m=1}^{n-1}m^{2H-2-l} dtds\\
&= \int_0^1 \int_0^1 \phi_k(t)\phi_k(s)  \sum_{l=2}^\infty {2H-2 \choose l}(t-s)^l \sum_{m=1}^{\infty}m^{2H-2-l} dtds\\
&=\int_0^1 \int_0^1 \phi_k(t)\phi_k(s)  \sum_{l=2}^\infty {2H-2 \choose l}(t-s)^l \zeta (l+2-2H) dtds,
\end{align*}
since the $m^{2H-2-l}$ are summable for $l\geq 1$.\\
In a similar manner, we get
\begin{align*}
\frac{1}{n} &\int_0^1 \int_0^1 \phi_k(t)\phi_k(s) \sum_{m=1}^{n-1}m\cdotp m^{2H-2}\left(\frac{t-s}{m}+1\right)^{2H-2}dtds\\
& \int_0^1 \int_0^1 \phi_k(t)\phi_k(s)  \sum_{l=2}^\infty {2H-2 \choose l}(t-s)^l \frac{1}{n}\sum_{m=1}^{n-1}m^{2H-1-l} dtds,
\end{align*}
which converges to zero, again, due to summability of $m^{2H-1-l}$.

In total, we conclude that the second summand in \eqref{eq:2} converges to
\[\int_0^1 \int_0^1 \phi_k(t)\phi_k(s)  \sum_{l=2}^\infty {2H-2 \choose l}(t-s)^l \zeta (l+2-2H) dtds,\]
and thus, with a symmetric calculation, the third summand tends to
\[\int_0^1 \int_0^1 \phi_k(t)\phi_k(s)  \sum_{l=2}^\infty {2H-2 \choose l}(s-t)^l \zeta (l+2-2H) dtds.\]
Adding up the two yields the desired result.
\end{proof}
The next proposition provides some additional information about $\bar{Z}_k$ and provides a more concise form for its variance.
\begin{proposition}
The variance of $\bar{Z}_k$ form the previous proposition can be simplified to
\[\frac{1}{\Gamma (2-2H)}\int_0^1\int_0^1 \phi_k(t)\phi_k(s)\int_0^\infty \frac{u^{1-2H}}{e^u-1}(e^{u(1-|t-s|)}+e^{u|t-s|}-2)du dt ds.\]
This expression is positive for all bounded non zero $L^2$-functions $\phi_k$ with zero integrals.
\end{proposition}
\begin{proof}
Our goal is to show that
\begin{align*}
\int_0^1 &\int_0^1 \phi_k (t)\phi_k (s)|t-s|^{2H-2}dt ds\\
&+ \sum_{l=1}^\infty 2 {{2H-2}\choose{2l}}\zeta (2l+2-2H) \int_0^1\int_0^1 \phi_k (t)\phi_k (s)(t-s)^{2l}dt ds
\end{align*}
can be rewritten in the above integral form. For the first summand the definition of Gamma function provides the representation
\[|t-s|^{2H-2}=\frac{1}{\Gamma (2-2H)} \int_0^\infty u^{1-2H}e^{-u|s-t|}du.\]
For the other summands we make use of the formula $\Gamma(z)\zeta (z)=\int_0^\infty \frac{u^{z-1}}{e^u-1}du$ for $z>1$ and rewrite them as follows:
\begin{align*}
2 & \sum_{l=1}^\infty {{2H-2}\choose{2l}}\zeta (2l+2-2H) \int_0^1\int_0^1 \phi_k (t)\phi_k (s)(t-s)^{2l}dt ds\\
&=2\int_0^1\int_0^1 \phi_k(t)\phi_k(s)\sum_{l=1}^\infty\frac{(2H-2)_{2l}}{(2l)!}\frac{1}{\Gamma (2l+2-2H)}\int_0^\infty \frac{u^{2l+1-2H}}{e^u-1}du (t-s)^{2l}ds dt\\
&=2\int_0^1\int_0^1 \phi_k(t)\phi_k(s)\int_0^\infty \sum_{l=1}^\infty \frac{(2H-2)_{2l}}{(2l)!\Gamma(2-2H)(2-2H)^{(2l)}}\frac{u^{2l+1-2H}(t-s)^{2l}}{e^u-1}dudsdt\\
&=\frac{1}{\Gamma (2-2H)}2 \int_0^1\int_0^1 \phi_k(t)\phi_k(s)\int_0^\infty \frac{u^{1-2H}}{e^u-1}\sum_{l=1}^\infty \frac{(u(t-s))^{2l}}{(2l)!}du ds dt\\
&=\frac{1}{\Gamma (2-2H)}2 \int_0^1\int_0^1 \phi_k(t)\phi_k(s)\int_0^\infty \frac{u^{1-2H}}{e^u-1}(\cosh (u(t-s))-1)du ds dt,
\end{align*}
where $(z)_{k}$ and $(z)^{(k)}$ denote the falling and rising factorials respectively. For even $k$ it follows from the definition that $(-z)_{k}=(z)^{(k)}$.

Recall that
\[\cosh(u(t-s))-1=\frac{e^{u(t-s)}+e^{u(s-t)}-2}{2}=\frac{e^{u|t-s|}+e^{-u|t-s|}-2}{2}\]
for any $t,\,s$ and add up the summands of the variance expression in order to obtain
\begin{align*}
\int_0^1 &\int_0^1 \phi_k (t)\phi_k (s)\frac{1}{\Gamma (2-2H)}\int_0^\infty u^{1-2H} \left(e^{-u|s-t|}+ \frac{2}{e^u-1}(\cosh(u(t-s))-1)\right)dudsdt\\
&=\int_0^1 \int_0^1 \phi_k (t)\phi_k (s)\frac{1}{\Gamma (2-2H)}\int_0^\infty u^{1-2H} \left(e^{-u|s-t|}+ \frac{e^{u|t-s|}+e^{-u|t-s|}-2}{e^u-1}\right)dudsdt\\
&=\int_0^1 \int_0^1 \phi_k (t)\phi_k (s)\frac{1}{\Gamma (2-2H)}\int_0^\infty \frac{u^{1-2H}}{e^u-1}(e^{u(1-|t-s|)}+e^{u|t-s|}-2)dudsdt,
\end{align*}
which was our claim.

Now let us prove that the obtained variance is indeed positive, thus confirming the rate of convergence suggested above. For elements of the real $L^2([0,\,1])$-Fourier basis this claim is shown (up to an application of Fubini's Theorem) in Proposition \ref{prop:3}. We also obtain from this proposition that in this particular case the variance simplifies to
\[\frac{1}{\Gamma (2-2H)}\int_0^\infty \frac{u^{2-2H}}{(2\pi n)^2+u^2}du\]
for $\phi_k(x)=\sqrt{2}\sin(2\pi n)$ or $\phi_k(x)=\sqrt{2}\cos(2\pi n)$.

An arbitrary $L^2$-function $\phi_k$ with zero integral can be written as $\sum_{n\in\mathbb Z\backslash \{0\}}c_n f_n$, where $f_n$ are elements of the Fourier basis without the constant component and we have for such a decomposition:
\begin{eqnarray*}
\lefteqn{\int_0^1  \int_0^1 \phi_k (t)\phi_k (s)\frac{1}{\Gamma (2-2H)}\int_0^\infty \frac{u^{1-2H}}{e^u-1}(e^{u(1-|t-s|)}+e^{u|t-s|}-2)dudsdt}&&\\
&=&\int_0^1 \int_0^1 \sum_{m,\,n\in\mathbb Z\backslash \{0\}}c_n f_n(t)c_m f_m(s)\frac{1}{\Gamma (2-2H)}\times\\
&&\int_0^\infty \frac{u^{1-2H}}{e^u-1}(e^{u(1-|t-s|)}+e^{u|t-s|}-2)dudsdt\\
&=&\sum_{m,\,n\in\mathbb Z\backslash \{0\}}c_mc_n\frac{1}{\Gamma (2-2H)}\times\\
&&\int_0^\infty \frac{u^{1-2H}}{e^u-1}  \int_0^1 \int_0^1 f_m (t)f_n (s)(e^{u(1-|t-s|)}+e^{u|t-s|}-2) dsdt du\\
&=&\sum_{n\in\mathbb Z\backslash \{0\}}c_n^2\frac{1}{\Gamma (2-2H)}\times\\
&&\int_0^\infty \frac{u^{1-2H}}{e^u-1}  \int_0^1 \int_0^1 f_n (t)f_n (s)(e^{u(1-|t-s|)}+e^{u|t-s|}-2) dsdt du,
\end{eqnarray*}
since all the off-diagonal terms disappear, as was demonstrated in Proposition \ref{prop:3}. We can now use the result for the Fourier basis  and complete the calculations:
\begin{align*}
\int_0^1 & \int_0^1 \phi_k (t)\phi_k (s)\frac{1}{\Gamma (2-2H)}\int_0^\infty \frac{u^{1-2H}}{e^u-1}(e^{u(1-|t-s|)}+e^{u|t-s|}-2)dudsdt\\
&=\frac{1}{\Gamma (2-2H)}\sum_{n\in\mathbb Z\backslash \{0\}}c_n^2 \int_0^\infty \frac{u^{2-2H}}{(2\pi n)^2+u^2}du,
\end{align*}
which is clearly positive if $\phi_k$ is nonzero.
\end{proof}
In the context of a different scaling for some of the components there is an additional remark to be made.
\begin{remark}
Since for each $k\in \{1,\dots ,p\}$ the term $M_1$ in $\hat{\mu}_k-\mu_k$ is Gaussian, if different components of the vector $\mu$ are weighted differently (depending on whether the corresponding $\phi_k$ have zero integrals), the whole vector converges jointly to a multivariate Gaussian. With a calculation similar to those in \ref{prop:1} one can show that the components with different weights are uncorrelated.
\end{remark}

\section{Appendix}
In this chapter we have collected some technical results that are used in the proofs of this paper.
\begin{lemma}
\label{l:4}
For a centred normal sequence $(X_n)_{n\in\N}$ of random variables we have: The squared $L^2$ norm of order at most $\frac{1}{n^\beta}$ for $\beta > 0$ implies almost sure convergence.
\end{lemma}
\begin{proof}
First note that the squared $L^2$ norm of a centred normal random variable is its variance. For $k\in\N$ the $2k$-th moment is completely determined by it; we have
\[\E [X_n^{2k}]=C_k\E[X_n^2]^k\lesssim \frac{1}{n^{\beta k}} \]
by assumption.
If we now check the summability criterion, this consideration allows us to get the result by Markov's inequality for $f(x)=x^{2k}$ and $k$ such that $\beta k>1$:
\[\sum_{n=1}^{\infty}P(|X_n|>\varepsilon)\leq \sum_{n=1}^{\infty}\frac{\E [X_n^{2k}]}{\varepsilon^{2k}}=\frac{1}{\varepsilon^{2k}}C_k\sum_{n=1}^{\infty}\E [X_n^2]^k\lesssim \sum_{n=1}^{\infty}\frac{1}{n^{\beta k}}<\infty. \]
\end{proof}
\begin{proposition}
\label{prop:2}
For $\alpha >0$ we have $\int_0^t e^{\alpha u} u^{2H-2}du\lesssim t^{2H-2}e^{\alpha t}$.
\end{proposition}
\begin{proof}
An analytic result from \cite{Abra} yields that the left-hand side is bounded by a constant times the right-hand side for large $t\in\mathbb R^+$. For smaller $t$, that is, for $t\leq t_0$ for some $t_0$, note that the left side is continuous while the right side has one discontinuity at $0$, where it tends to infinity. Therefore, it is also possible to find a constant for which the bound holds on the compact interval $[0,\,t_0]$. By taking the maximum of the two we obtain the result.
\end{proof}
\begin{proposition}
\label{prop:3}
Let $(f_n)_{n\in\mathbb Z\backslash \{0\}}$ be the real $L^2([0,\,1])$-Fourier basis without the constant element, i.e. $f_n(x)=\sqrt{2}\sin (2\pi n x)$ and $f_{-n}(x)=\sqrt{2}\cos (2\pi n x)$ for $n\in\mathbb N$. Then for any $u>0$ the integral
\[\int_0^1\int_0^1 f_n(t)f_m(s)(e^{u(1-|t-s|)}+e^{u|t-s|}-2)dtds\]
is positive and equal to $\frac{2(e^u-1)u}{(2\pi n)^2+u^2}$ if $m=n$ and zero otherwise.
\end{proposition}
\begin{proof}
Let us wirte $z=e^u$ and calculate for $m,\,n\in\mathbb Z\backslash \{0\}$:
\begin{align*}
\int_0^1 &\int_0^1 f_n(t)f_m(s)(z^{(1-|t-s|)}+z^{|t-s|}-2)dtds\\
&=\int_0^1\int_{t-1}^t f_n(t)f_m(t-v)(z^{1-|v|}+z^{|v|})dvdt\\
&=\int_0^1 f_n(t) \int_{t-1}^t f_m(t-v)(z^{1-|v|}+z^{|v|})dvdt.
\end{align*}
By classical trigonometric identities we can decompose $f_m(t-v)$ as
\[\sqrt{2}f_m(t-v) =f_m(t)f_{-m}(v)-f_{-m}(t)f_m(v)\]
if $m$ is positive and
\[\sqrt{2}f_m(t-v)=f_m(t)f_m(v)+f_{-m}(t)f_{-m}(v)\]
if $m$ is negative. Thus, for the second part of the statement it suffices to show that the integral
\[\int_{t-1}^t f_m(v)(z^{1-|v|}+z^{|v|})dv\]
is independent of $t$ for all $m\in \mathbb Z\backslash \{0\}$. This is indeed the case, because
\[\int_{t-1}^0f_{m}(v)(z^{1+v}+z^{-v})dv=\int_t^1 f_m(v)(z^{1-v}+z^v), \]
and therefore,
\[\int_{t-1}^t f_m(v)(z^{1-|v|}+z^{|v|})dv=\int_{0}^1 f_m(v)(z^{1-v}+z^{v})dv\]
is indeed independent of $t$. For symmetry reasons the integral vanishes for $m>0$.

If $n=m$, the same trigonometric identities can be used to show that
\[\int_0^1\int_0^1 f_n(t)f_n(s)(e^{u(1-|t-s|)}+e^{u|t-s|}-2)dtds=\frac{1}{\sqrt{2}}\int_0^1 f_{-n}(v)(z^{1-v}+z^v)dv\]
if $n$ is positive and
\[\int_0^1\int_0^1 f_n(t)f_n(s)(e^{u(1-|t-s|)}+e^{u|t-s|}-2)dtds=\frac{1}{\sqrt{2}}\int_0^1 f_{n}(v)(z^{1-v}+z^v)dv\]
if $n$ is negative. Since
\[\int_0^1 \cos(2\pi n v)(z^{1-v}+z^v)dv=\frac{2(z-1)\log(z)}{(2\pi n)^2+(\log(z))^2}=\frac{2(e^u-1)u}{(2\pi n)^2+u^2} \]
is positive for all $u>0$, the first part of the claim is proved.
\end{proof}

\end{document}